\documentclass [11pt]{article}
\usepackage{amsthm}
\usepackage{amssymb}
\usepackage{epsfig}
\usepackage{amsthm}
\usepackage{srcltx}
\usepackage[dvipsnames,usenames]{color}
\usepackage{paralist}
\usepackage{amsmath}
\usepackage{amsfonts}

\newcounter{contador}
\newtheorem{propo}[contador]{Proposition}
\newtheorem{teo}[contador]{Theorem}
\newtheorem{lem}[contador]{Lemma}

\newtheorem{corol}[contador]{Corollary}
\newtheorem{nota}[contador]{Remark}

\newcommand{\g}{\gamma}

\newcommand{\su}{{\mathbb S}^1} 

\newcommand{\R}{{\mathbb R}}
\newcommand{\C}{{\mathbb C}}
\newcommand{\N}{{\mathbb N}}
\newcommand{\Q}{{\mathbb Q}}

\newcommand{\V}{{\cal{V}}}
\newcommand{\U}{{\cal{U}}}

\title{Non-integrability of measure preserving maps\\ via Lie symmetries\footnote{{\bf Acknowledgements.}
The authors are supported by Ministry of Economy and Competitiveness
of the Spanish Government through grants MTM2013-40998-P (first and
second authors) and DPI2011-25822 (third author). The first  and
second authors are also supported by the grant 2014-SGR-568  from
AGAUR, Generalitat de Catalunya  and  BREUDS project
FP7-PEOPLE-2012-IRSES-318999. The third author is  supported by the
grant  2014-SGR-859  from AGAUR, Generalitat de Catalunya.}}

\author{Anna Cima$^{(1)}$, Armengol Gasull$^{(1)}$ and V\'{\i}ctor Ma\~{n}osa $^{(2)}$
  \\*[.1truecm]
{\small \textsl{$^{(1)}$ Departament de Matem\`{a}tiques, Facultat
de Ci\`{e}ncies,}}
\\*[-.25truecm] {\small \textsl{Universitat Aut\`{o}noma de Barcelona,}}
\\*[-.25truecm] {\small \textsl{08193 Bellaterra, Barcelona, Spain}}
\\*[-.25truecm] {\small \textsl{cima@mat.uab.cat,
gasull@mat.uab.cat}}\\
\\*[-.25truecm] {\small \textsl{$^{(2)}$ Departament de Matem\`{a}tica Aplicada III,}}
\\*[-.25truecm] {\small \textsl{Control, Dynamics and Applications Group (CoDALab)}}
\\*[-.25truecm] {\small \textsl{Universitat Polit\`{e}cnica de Catalunya}}
\\*[-.25truecm] {\small \textsl{Colom 1, 08222 Terrassa, Spain}}
\\*[-.25truecm] {\small \textsl{victor.manosa@upc.edu}}}
\begin{document}

\maketitle
\begin{abstract}
We consider the problem of characterizing, for certain natural
number $m$,  the local $\mathcal{C}^m$-non-integrability near
elliptic fixed points of smooth planar measure preserving maps.  Our
criterion relates this non-integrability with the existence of some
Lie Symmetries associated to the maps, together with the study of
the finiteness of its periodic points. One of the steps in the proof
uses the regularity of the  period function on the whole period
annulus for non-degenerate centers, question that we believe that is
interesting by itself. The obtained criterion can be applied to
prove the local non-integrability of the Cohen map and of  several
rational maps coming from second order difference equations.
\end{abstract}



\section{Introduction and main results}\label{S-intro}

In the last years the development of criteria to determine the integrable nature of discrete
dynamical systems has  been the focus of an  intensive research activity (see \cite{GR1} and
references therein), however there are very few non-integrability results for discrete dynamical
systems, see for instance \cite{CGM13a,CS, CZ,DR,RT, St} and their references. The main result of
this paper, Theorem \ref{T-Preserva-mesura} below,  provides  a criterion to establish the local
non-integrability of real planar measure preserving maps in terms of non existence of local first
integrals of class $\mathcal{C}^m$, for certain $m\in\N$, near an \emph{elliptic} fixed point
(that is, a fixed point such that the eigenvalues of the associated linear part lie in the unit
circle, but excluding the values $\pm 1$).

We will say that a planar map is $\mathcal{C}^m$-\emph{locally integrable at an elliptic fixed
point $p$} if it does exist a neighborhood $\U$ of $p$ and  a locally non-constant real valued
function $V\in{\mathcal{C}}^m(\U)$, with $m\geq 2$, (called first integral) such that
$V(F(x))=V(x),$ all the level curves $\{V=h\}\cap \U$ are closed curves surrounding $p$ and,
moreover,  $p$ is an isolated non-degenerate critical point of $V$ in~$\U$.

Prior to state the main result, we recall that a map $F$ defined on $\U$, an open set of $\R^2$,
preserves an  absolutely continuous measure with respect the Lebesgue's one  with non-vanishing
density $\nu$,  if $m(F^{-1}(B))=m(B)$ for any measurable set $B$, where $ m(B)=\int_{B}
\nu(x,y)\,dxdy$, and $\left.\nu\right|_{\mathcal{U}}\ne0$. For the sake of simplicity, in this
paper sometimes we will refer these maps simply as   \emph{measure preserving maps}.

When  the eigenvalues $\lambda,\bar \lambda=1/\lambda$ of the linear
part of a $\mathcal{C}^{1}$-planar map $F$ at an elliptic fixed
point $p \in\R^2$ are not roots of unity of order $\ell$ for
$0<\ell\leq k$ we will say that $p$ is {\it not $k$-resonant}.
Recall that for not $k$-resonant elliptic fixed points  a
$\mathcal{C}^{k}$-map, $F$, is locally conjugated to its {\it
Birkhoff normal form} plus some remainder terms, see \cite{AP}:
\begin{equation}\label{E-FN}
F_{B}(z)=\lambda z\left(1+\sum\limits_{j=1}^{[(k-1)/2]} B_j (z\bar{z})^j\right)+o(|z|^k),
\end{equation}
where $z=x+iy$, and $[\cdot]$ denotes the integer part.  It is
well-known  that near a locally integrable elliptic point  the first
non-vanishing \emph{Birkhoff constant} $B_n$, if  exists, must be
purely imaginary.  We recall a proof of this fact in
Lemma~\ref{L-dynamicaFN-gen}.

The main result of this paper is the following theorem:

\begin{teo}\label{T-Preserva-mesura}  Let $F$ be  a  $\mathcal{C}^{2n+2}$-planar map defined on an open set
$\U\subseteq\R^2$ with  an elliptic   fixed point $p$,  not $(2n+1)$-resonant, and such that its
first non-vanishing Birkhoff constant is  $B_n= i\,b_n$,  for some $0<n\in\N$ and
$b_n\in\R\setminus\{0\}.$ Moreover, assume that $F$ is a measure preserving map with a
non-vanishing density $\nu\in\mathcal{C}^{2n+3}$. If, for an unbounded sequence of natural numbers
$\{N_k\}_k$, $F$ has finitely many $N_k$-periodic points in $\U$ then it is not $\mathcal{C}^{
2n+4}$-locally integrable at~$p$.
\end{teo}

Our proof uses some of the ideas presented by G.~Lowther in \cite{L} for explaining the
non-integrability of the Cohen map. As we will see, our result has also several applications for
proving non-smooth integrability of several rational difference equations.

One of the main ingredients in our proof of
Theorem~\ref{T-Preserva-mesura}  is that any integrable measure
preserving map has an associated vector field $X$, called a
\emph{Lie Symmetry}, such that $F$ can be expressed in terms of the
flow of $X$, see Section \ref{SS-prova} for further details. As we
will see, to proceed with our approach, from this Lie symmetry we
need to construct another one, say $Y$, having an isochronous
center. Our construction of this  vector field $Y$ is based on the
study of the regularity of the so called {\it period function} in a
neighborhood of a non-degenerate center. Let us recall its
definition.

Let $p$ be a non-degenerate center of a smooth vector field $X$, that is such that $DX(p)$ has
eigenvalues $\pm \,i \omega$ with $0\ne\omega\in\R$. Let  $\V$ be the largest neighborhood of $p$
such that $\V\setminus\{p\}$  is foliated by periodic orbits. This set is usually called {\it
period annulus}. Then for all $(x,y)\in\V\setminus\{p\}$, the function $T(x,y)$ giving the
(minimal) period of the closed orbit passing through $(x,y)$ can be extended continuously to $p$
as $T(0,0)=2\pi/\omega$. As usual, we will call this function  $T$, defined on the whole set $\V$,
the {\it period function} of $X$ on $\V$.

The regularity of the period function on $\V$ for non-degenerate
centers of $\mathcal{C}^\infty$ or analytic  planar vector field is
known. It coincides, in the whole set $\V$, with the regularity of
the corresponding vector field, see~\cite{V}. In next result we show
that this is no more true for $\mathcal{C}^k$-vector fields,
$k\in\N$.

\begin{teo}\label{T-regularitat-periode}
Let $X$ be a $\mathcal{C}^k$-vector field with $1\le k\in\N\cup\{\infty,\omega\}$ with a
non-degenerate center $p$, and let $\V$ be its period annulus. Then the period function $T$ is of
class $\mathcal{C}^{k}$ on $\V\setminus\{p\}$ and, at $p$, it is of class~$\mathcal{C}^{k-1},$
where  for the sake of notation $\infty-1=\infty$ and $\omega-1=\omega$. Moreover, in general, the
regularity of $T$ at $p$ can not be improved.
\end{teo}

Notice that since the period function $T$ of a non-degenerate center
on its period annulus is clearly a first integral for the
corresponding vector field, a direct consequence of the above result
is:
\begin{corol}\label{C-first-integral}
Let $p$ be  a non-degenerate center of a $\mathcal{C}^k$-vector
field, $1\le k\in\N\cup\{\infty,\omega\}$, and let $\V$ be its
period annulus. Then the vector field has a
$\mathcal{C}^{k-1}$-first integral on $\V$.
\end{corol}

In fact, it is already known that if $p$ is a center, no necessarily
non-degenerate,  of a $\mathcal{C}^k$-vector field
($k\in\N\cup\{\infty\}$), then there exists a $\mathcal{C}^k$-first
integral in a small enough neighborhood of $p$, see~\cite{MS}.
Nevertheless, although the corollary gives a much weaker result, our
approach is different to the one of~\cite{MS}.

The second ingredient is a method for checking when the discrete dynamical system generated by a
map $F:\R^M\to\R^M$ has finitely many $K$-periodic points. Or, equivalently, when the system
\[
{\bf x}_1-F({\bf x}_0)={\bf 0},\,{\bf x}_2-F({\bf x}_1)={\bf
0},\ldots,{\bf x}_{K-1}-F({\bf x}_{K-2})={\bf 0},\,{\bf x}_0-F({\bf
x}_{K-1})={\bf 0},\,
\]
has finitely many real solutions. Notice that the above system can
be written in a compact for as $\widehat G({\bf y})={\bf 0},$ where
${\bf y}=({\bf x}_0,{\bf x}_1,\ldots,{\bf x}_{K-1})\in\R^N,$ for
some map $\widehat G:\R^N\rightarrow\R^N$, where $N=K\,M.$ We prove
the following result, that can be applied in case that all solutions
of the system $\widehat G({\bf y})={\bf 0}$ are also solutions of a
new system, $G({\bf y})={\bf 0}$, for some polynomial map
$G:\R^N\rightarrow\R^N$.

\begin{teo}\label{teog}
Let $G:\C^N\rightarrow\C^N$ be a polynomial map of degree $d$. Let $G_d$ denote the homogenous map
corresponding to the degree $d$ terms of $G$. If ${\bf y}={\bf 0}$ is the unique solution in
$\C^N$ of the homogeneous system $G_d({\bf y})={\bf 0}$, then the polynomial system  $G({\bf
y})={\bf 0}$ has finitely many solutions.
\end{teo}

Although we have not found the above result in the literature, most probably it is a folklore
result. In any case, we sketch a proof in Section~\ref{SS-continua}.

Notice also, that by Bezout's Theorem, we also know that when the hypotheses of the theorem are
satisfied  the maximum number of solutions of  $G({\bf y})={\bf 0}$ is $d^N.$ Finally observe that
applying Theorem~\ref{teog} when $G$ is linear, that is $d=1$ and $G({\bf y})=A\,{\bf y}+{\bf b}$,
for some $N\times N$ matrix $A$, then $G_d({\bf y})=A\,{\bf y}$ and the condition  that $G_d({\bf
y})={\bf 0}$ if and only if ${\bf y}={\bf 0}$ reduces to $\det(A)\ne0.$ Therefore the above result
can be thought as a natural extension of the well-known result: system $A\,{\bf y}+{\bf b}={\bf
0}$ has finitely many solutions (in fact, 0 or 1) if $\det A\ne0.$

As a first application of Theorems~\ref{T-Preserva-mesura} and~\ref{teog} we recover the result of
G.~Lowther about the non-integrability of the Cohen map
\begin{equation}\label{E-Cohen}
    F(x,y)=\left(y,-x+\sqrt{y^2+1}\right),
\end{equation}
exposed in \cite{L}.
\begin{teo}\label{T-Cohen}
The Cohen map is not { $\mathcal{C}^6$}-locally integrable at its
fixed point $\left(\sqrt{3}/3,\sqrt{3}/3\right)$.
\end{teo}

According to   M.~Rychlik and M.~Torgesson \cite{RT},  the question about the integrability of
this map was first conjectured by H.~Cohen and comunicated by C.~de Verdi\`ere to J.~Moser in
1993. Rychlik and Torgesson showed that it has not a first integral given by algebraic functions.
Nowadays this map is considered unlikely to be integrable since numerical explorations show that
it has hyperbolic periodic points and chains of islands of period $14$ and $23$,
see~\cite{L,PABV}.

Our second application covers a wide class of rational difference equations. Consider
\begin{equation}\label{fyx}
    F(x,y)=\left(y,\frac{f(y)}{x}\right),
\end{equation}
where  $f=P/Q$ is a rational map. For the sake of notation,  define $\deg(f)=\deg(P)-\deg(Q)$. Its
fixed points are $p=(\bar{x},\bar{x})$, where $\bar{x}$ are the non-zero solutions of the equation
$f(\bar{x})=\bar{x}^2$ and $p$ is an elliptic point if and only if
$\left|{f'(\bar{x})}/{\bar{x}}\right|<2.$ Moreover~(\ref{fyx}) preserves the measure with
 density  $\nu(x,y)=1/(xy),$ that does not vanish on a neighborhood of the fixed points.

\begin{teo}\label{P_racionals2}
Consider the map (\ref{fyx}), where $f$ is a rational function with
$\deg(f)>2$.  If $p$ is an elliptic fixed point, not
$(2n+1)$-resonant, and such that its first non-vanishing Birkhoff
constant is  $B_n= i\,b_n$, for some $0<n\in\N$ and
$b_n\in\R\setminus\{0\},$ then $F$ is not ${
\mathcal{C}^{2n+4}}$-locally integrable at $p$.
\end{teo}

It is interesting to notice that when $\deg(f)\le 2$ there are integrable cases at least for
$\deg(f)\in\{-1,0,1,2\}.$ For $k\in\{-1,0,1\}$ it suffices to consider the periodic maps $F$ with
$f(y)=y^k,$ because all rational periodic maps are rationally integrable, see~\cite{CGM}. Other
integrable, non-periodic cases  are  the well-known Lyness map, that corresponds to $f(y)=(a+y)$,
see~\cite{D}, or for $\deg(f)=2$, the map studied by G.~Bastien and M.~Rogalski in \cite{BR},
given by $f(y)=(a-y+y^2),$ which possesses the first integral $V(x,y)=(x^2+y^2-x-y+a)/(xy)$.

In the case with $\deg(f)=2$, we study with more detail the family of maps

\begin{equation}\label{BRgen}
    F(x,y)=\left(y,\frac{A+By+Cy^2}{x}\right),\quad C\ne0,
\end{equation}
that extends the one given in~\cite{BR}. In next result we prove
that in this family integrability and non-integrability coexist.

\begin{propo}\label{familybracfacil}
A map (\ref{BRgen})  having   an elliptic   fixed point $p$,  not
$5$-resonant, is ${\mathcal{C}^{6}}$-locally integrable at $p$ if
and only if $C=1.$ Moreover, when $C=1$ the map has the rational
first integral $V(x,y)=(x^2+y^2+B(x+y)+A)/(xy).$
\end{propo}

The third application deals with the area preserving maps (density $\nu=1$),
\begin{equation}\label{apm}
    F(x,y)=\left(y,-x+f(y)\right),
\end{equation}
with  $f$ also a rational map. In this case we prove:

\begin{teo}\label{P_racionals3}
Consider the map (\ref{apm}), where $f$ is a rational function with $\deg(f)>1$.  If $p$ is an
elliptic fixed point, not $(2n+1)$-resonant, and such that its first non-vanishing Birkhoff
constant is  $B_n= i\,b_n$, for some $0<n\in\N$ and $b_n\in\R\setminus\{0\},$ then $F$ is not
$\mathcal{C}^{2n+4}$-locally integrable at $p$.
\end{teo}

Also for this family,  when $\deg(f)\le 1$, there are integrable
cases   at least for $\deg(f)\in\{0,-1\}$. For  $\deg(f)=0$,
$f(y)\equiv k$ and   the map is an involution and therefore
rationally integrable, see again~\cite{CGM}. For $\deg(f)=-1$ we can
consider the well-known integrable McMillan-Gumowski-Mira map (for
short MGM map) with $f(y)=ay/(1+y^2)$ and first integral
$V(x,y)=x^2y^2+x^2+y^2-axy$, see~\cite{Gu,McMillan}. We remark that
this map possesses elliptic not resonant points for many values of
$a$, see Section~\ref{SS-Birk}.

Finally, as a consequence   of Theorems~\ref{P_racionals2} and
\ref{P_racionals3}, we prove that two celebrated integrable maps,
the MGM and the Lyness ones given above, are ``isolated" in a
suitable set of rational maps.

\begin{corol}\label{pertur} Let $g$ be a rational function. Consider the
maps:
\begin{enumerate}[(i)]
\item $
  F_\varepsilon(x,y)=\left(y,-x+\frac{a
  y}{1+y^2}+\varepsilon\,g(y)\right),$
with $\deg(g)>1$ and $a\in (-2,\infty)\setminus \{-1,0,2\}.$

\item
$
G_\varepsilon(x,y)=\left(y,\frac{a+y+\varepsilon\,g(y)}{x}\right),$
with  $\deg(g)>2$ and $a\in (-1/4,\infty)\setminus \{0,1\}.$
\end{enumerate}
 Then,
for $|\varepsilon|$ small enough,  $F_\varepsilon$ or
$G_\varepsilon$  are $\mathcal{C}^{6}$-locally integrable at its
corresponding elliptic fixed points if and only if $\varepsilon=0$.
\end{corol}

Similarly, we obtain:

\begin{corol}\label{pertur2} Let $g$ and $h$ be  rational functions with
either $\deg(g)>1$ or $\deg(h)>2,$ and $2\deg(g)\ne \deg(h)$. Then, for $|\varepsilon|$ small
enough, the map
\[
H_\varepsilon(x,y)=\left(y,-x+\sqrt{y^2+1+\varepsilon h(y)}+\varepsilon\,g(y)\right),\]
 is not $\mathcal{C}^{6}$-locally integrable at
its elliptic fixed point.
\end{corol}

We remark that a similar result to the one of item (i) of Corollary
\ref{pertur} was obtained in \cite{DR} when $a>2$, proving that the
map $F_\varepsilon$, when $0\ne |\varepsilon|$ is small enough, is
not holomorphically integrable near the homoclinic loops passing
through the origin. When $a<2$ these homoclinic loops no more exist.
Alternatively, our approach proves the smooth non-integrability near
the elliptic points, that exist for all $a\in
(-2,\infty)\setminus\{0,2\}.$ The value $a=-1$ is excluded in our
study because for it the map $F_0$ has a 2-resonance  at the origin.

The paper is structured as follows. Section~\ref{S-proofs} contains several preliminary results.
More concretely, in Section~\ref{SS-prelim} we include our results about the regularity of the
period function and in particular the proof of Theorem~\ref{T-regularitat-periode}. In
Section~\ref{SS-continua} we prove Theorem~\ref{teog}, which recall  that gives a tool for proving
the existence of finitely many $N$-periodic points and in Section~\ref{SS-Birk} we introduce the
Birkhoff constants and compute them in some simple examples that will appear afterwards. Section
\ref{S-Cohen} is devoted to prove   the non-integrability of the Cohen map, that is Theorem
\ref{T-Cohen}.  Finally, in Section~\ref{S-Other}, we apply our results to several rational maps
motivated from well-known difference equations, proving Theorems~\ref{P_racionals2}
and~\ref{P_racionals3} and their consequences.

\section{Preliminary results}\label{S-proofs}

\subsection{Regularity of the period function.
Proof of Theorem \ref{T-regularitat-periode}}\label{SS-prelim}

Observe that from the implicit function theorem, and the regularity
of the flow of $X$, the function
 $T(x,y)$ is of class $\mathcal{C}^k$ in $\V\setminus\{p\}$. So, to prove
 Theorem~\ref{T-regularitat-periode} it only remains to
 study the regularity of $T$ at $p$.
 In the cases $k\in\{\infty,\omega\}$ this is already done in~\cite{V}. Let us consider $k\in\N.$
 It is well-known that, since $p$ is a non-degenerate center, the vector
field in a neighborhood of $p$ is $\mathcal{C}^k$-conjugated to its
Poincar\'{e} normal form , see~\cite{GH}. Its corresponding
differential equation is
$$
\dot{z}=iz\,\left(\omega +\sum\limits_{j=1}^{s} a_{2j} (z\bar{z})^j\right)+o(|z|^k).
$$
where $z=x+iy\in\C$,  $\omega, a_{2j}\in\R$, and
$s=\left[(k-1)/2\right]$. Taking polar coordinates we get:
$$
    \left\{
      \begin{array}{ll}
        \dot{r}=o_{\theta}(r;k),\\
        \dot{\theta}=\omega +\sum\limits_{j=1}^{s}
a_{2j} r^{2j}+o_{\theta}(r;k-1),
      \end{array}
    \right.
$$
where $o_{\theta}(r;m)$ stands for a $\mathcal{C}^m$ function of the form $f(r,\theta)\,r^m$ such
that $\lim\limits_{r\to 0} f(r,\theta)=0$, uniformly in $\theta$. Notice that all the derivatives
of $o_{\theta}(r;m)$ of order less or equal than $m$ at the origin are zero. Hence
$$
\frac{dr}{d\theta}=\frac{o_{\theta}(r;k)}{\omega +\sum\limits_{j=1}^{s} a_{2j}
r^{2j}+o_{\theta}(r;k-1)}=o_\theta(r;k).
$$
The solution of the above equation with initial condition $r(\alpha)=\rho$, is
$r(\theta;\rho,\alpha)=\rho+o_{\theta,\alpha}(\rho;k)$, where here $o_{\theta,\alpha}(r;m)$ stands
for a $\mathcal{C}^m$ function of the form $r^m\, g(r,\theta,\alpha)$ and $\lim\limits_{r\to 0}
g(r,\theta,\alpha)=0$, uniformly in both variables $\theta$ and $\alpha$. Moreover, the period
function expressed in the  polar coordinates $(\rho,\alpha)$, is
$$\begin{array}{rl}
    \widetilde{T}(\rho,\alpha) & =\displaystyle{\int_{\alpha}^{\alpha+2\pi} \frac{d\theta}{\omega
+\sum\limits_{j=1}^{s} a_{2j}
r^{2j}(\theta;\rho,\alpha)+o_{\theta,\alpha}(r(\theta;\rho,\alpha);k-1)}}=\\
&=\displaystyle{\int_{\alpha}^{\alpha+2\pi} \frac{d\theta}{\omega +\sum\limits_{j=1}^{s} a_{2j}
\rho^{2j}+o_{\theta,\alpha}(\rho;k-1)}}\\
 &=\displaystyle{\int_{\alpha}^{\alpha+2\pi}
\frac{1}{\omega} +\sum\limits_{j=1}^{s} \tau_{2j}
\rho^{2j}+o_{\theta,\alpha}(\rho;k-1) \,d\theta}\\
     &=\frac{2\pi}{\omega}
+\sum\limits_{j=1}^{s} T_{2j} \rho^{2j}+\int_{\alpha}^{\alpha+2\pi} o_{\theta,\alpha}(\rho;k-1) \,
d\theta,
  \end{array}
$$
where $T_{2j}$ are some real constants. Notice also that
\[
\int_{\alpha}^{\alpha+2\pi} o_{\theta,\alpha}(\rho;k-1) \, d\theta= o_{\alpha}(\rho;k-1),
\]
because  if
$o_{\theta,\alpha}(\rho;k-1)=\rho^{k-1}g(\rho,\theta,\alpha),$ then
the function $g(\rho,\theta,\alpha)$ tends to zero, when $\rho$ goes
to zero, uniformly in $\theta$ and $\alpha$, and therefore,
$$
\widetilde{T}(\rho,\alpha)=\frac{2\pi}{\omega}+\sum\limits_{j=1}^{s} T_{2j}
\rho^{2j}+\rho^{k-1}h(\rho,\alpha),
$$
for some $h$ such that  $\lim\limits_{\rho\to 0} h(\rho,\alpha)=0$, uniformly in $\alpha$.  Then,
the period function at the point $z=x+iy$ is $T(x,y)=\widetilde{T}(\sqrt{x^2+y^2},\arg(x+iy))$, so
$$
T(x,y)= \frac{2\pi}{\omega}+\sum\limits_{j=1}^{s} T_{2j}
(x^2+y^2)^{j}+(x^2+y^2)^{\frac{k-1}{2}}H(x,y),
$$
with  $H(x,y)=h(\sqrt{x^2+y^2},\arg(x+iy))$, satisfying $ \lim\limits_{(x,y)\to (0,0)} H(x,y)=0.$
This implies that $(x^2+y^2)^{\frac{k-1}{2}}H(x,y)$ is of class $\mathcal{C}^{k-1}$ at the origin,
with all the derivatives zero, as we wanted to show.

Finally we  give some examples which prove that the regularity given at $p$ can not be improved.
For $ a \in\R$ consider the vector field, with associated differential system,
$$
    \left\{
      \begin{array}{ll}
        \dot{x}=-y\left(1+(x^2+y^2)^{ a }\right),\\
        \dot{y}=x\left(1+(x^2+y^2)^{ a }\right),
      \end{array}
    \right.
$$
Its period function is
$$
T(x,y)=\frac{2\pi}{1+(x^2+y^2)^{ a }}.
$$
Taking $ a =k/2$ when $k$ is odd, or $ a =k/(k+1)$ when $k$ is even, we obtain
$\mathcal{C}^{k}$-vector fields such that its corresponding period function is of class
$\mathcal{C}^{k-1}$, and not of class $\mathcal{C}^{k}$ at the origin.

\subsection{Existence of finitely many periodic points. Proof of Theorem \ref{teog}}\label{SS-continua}

We start by proving Theorem \ref{teog}. Then, in Proposition
\ref{L-sistemaN-OP},   we adapt it for maps coming from difference
equations.

\begin{proof}[Proof of Theorem \ref{teog}] Set ${\bf x}=(x_0,\ldots,x_{N-1})$.  We want to prove that
when $\mathbf{0}$ is the unique solution of $G_d({\bf x})={\bf 0}$ then  $G({\bf x})={\bf 0}$ has
 finitely many solutions. Observe that the system $G({\bf x})={\bf 0}$ defines an
algebraic set ${\mathcal X}\subset\C P^N$ given by $\tilde{G}({{\bf z}})={\bf 0},$ where $\tilde
G= (\tilde{G}_0, \ldots,
 \tilde{ G }_{N-1},x_N^d),$
$\tilde{ G }_i({{\bf z}})$ is the corresponding homogenization of $ G _i$, and ${{\bf
z}}=[x_0:\ldots:x_{N-1}:x_N]\in\C P^N$ where $\{x_N=0\}$ is the hyperplane at infinity.

We  will use that given an algebraic set ${\mathcal X}$ such that
$\dim({\mathcal X})=r$, then $\dim({\mathcal X}\cap\{x_N=0\})\geq
r-1$. This fact is a consequence of the well-known Projective
Dimension Theorem on dimension theory of algebraic varieties,
see~\cite[Th. 7.2]{Harst}, also ~\cite[Cor. 3.14]{M}.

Assume that $G({\bf x})={\bf 0}$ has infinitely many solutions. Then  $\dim({\mathcal X})\ge1$. By
the  above result, $\dim({\mathcal X}\cap\{x_N=0\})\geq 0.$ This inequality implies that $\mathcal
X$ intersects the hyperplane of infinity. This fact is precisely equivalent to say that $G_d({\bf
x})={\bf 0}$ has some non-zero solution, as we wanted to prove.~\end{proof}

Another proof of Theorem~\ref{teog} holds by using the following consequence of Chevalley's
Theorem (see \cite[Ex. 3.19]{Harst}): the first coordinate projection of ${\mathcal X}$ is either
finite or dense in~$\C.$ Then, under our hypotheses, this projection is dense in $\C$, implying
that ${\mathcal X}$ reaches infinity at some points that  produce non-zero solutions of $G_d({\bf
x})={\bf 0}.$ Clearly,  Theorem~\ref{teog} does not hold for real algebraic varieties, as the
circle in $\R^2$, $x_1^2+x_2^2=1,$ shows.

Given a difference equation such that its periodic solutions satisfy
certain \emph{algebraic} recurrence $g(x_{n},\ldots,x_{n+k})=0$. The
existence of $N$-periodic orbits can be characterized by a system of
$N$ algebraic equations
\begin{equation}\label{E-sistemaN-OP}
G({\bf x})={\bf 0},\quad\mbox{where}\quad
G=(g_0,g_1,\ldots,g_{N-1}),
\end{equation}
${\bf x}=(x_0,\ldots,x_{N-1})$ and $ g_i({\bf x}):=g(x_{i\mbox{ mod } N},\ldots,x_{i+k\mbox{ mod }
N})$. If $d=\deg(g)$ then $\deg(G)=d$ and $G_d=(g_{0,d},g_{1,d},\ldots,g_{N-1,d}),$ where
$g_{i,d}$ is the homogenous part of degree $d$ of $g_i,$
\begin{equation}\label{Notacio}
g_{i,d}({\bf x}):=g_d(x_{i\mbox{ mod } N},x_{i+1\mbox{ mod }
N},\ldots,x_{i+k\mbox{ mod } N}).
\end{equation}

Therefore, Theorem \ref{teog} for difference equations reads as
follows:

\begin{propo}\label{L-sistemaN-OP}
Consider a difference equation of order $k$ such that its associated sequence  satisfies an
algebraic recurrence $ g(x_{n},\ldots,x_{n+k})=0$ of degree $d$. Then it has finitely many
$N$-periodic solutions   if  ${x}=\mathbf{0}$ is the unique solution  in $\C^N$ of the system
\begin{equation}\label{E-sistemaN-OP-homo}
  g_{0,d}({\bf x})=0\,,\, g_{1,d}({\bf x})=0\,,\,\ldots\,,\,
 g_{N-1,d}({\bf x})=0.
\end{equation}
\end{propo}

\subsection{Birkhoff normal forms.}\label{SS-Birk}

The computation of the Birkhoff normal form is a well-known issue
and we address  the reader to~\cite{AP} for further references. In
particular the expression of  the first Birkhoff constant of a map
with a not $3$-resonant fixed point at the origin,
\begin{equation}\label{E-F-diagonal}
F(x,y)=\Big(\lambda x\,+\sum_{i+j\geq 2} f_{i,j}x^iy^j, \frac{1}{\lambda}\,y\,+\sum_{i+j\geq 2}
g_{i,j}x^iy^j\Big),\end{equation} where $\lambda\in\C,$ $|\lambda|=1$, is
\begin{equation}\label{E-B1-Gen}
B_1=\frac{ \mathcal{P}_1(F) }{{\lambda}^{2} \left( \lambda-1 \right) \left( {\lambda}^{2}+\lambda+
1 \right) },
\end{equation}
where
$$\begin{array}{rl}
         \mathcal{P}_1(F)=& \left( f_{11}g_{11}+f_{21} \right)
{\lambda}^{4}-f_{11}
 \left( 2f_{20}-g_{11} \right) {\lambda}^{3}+ \left( 2
f_{02}g_{20} -f_{11}f_{20}+f_{11}g_{11} \right){\lambda}^{2}\\
&- \left( f_{11}f_{20}+f_{21} \right) \lambda+ f_{11}f_{20},
        \end{array}
$$ see for instance~\cite[Sect. 4]{CGM13}. The general expression of $B_2$ is
$$B_2=\frac{\mathcal{P}_3(F)}{{\lambda}^{4} \left( \lambda-1 \right) ^{3} \left( {\lambda}^{2}+
\lambda+1 \right) ^{3} \left( {\lambda}^{2}+1 \right)  \left( \lambda+ 1 \right) },
$$
where $\mathcal{P}_3(F)$ is a huge polynomial expression that can be found in \cite[App.
A]{CGM13}.

A well-known result is the following:

\begin{lem}\label{L-dynamicaFN-gen}
For $1<n\in\N$, consider a $\mathcal{C}^{2n+2}$-map $F$  with an elliptic fixed point $p\in
\mathcal{U}$,  not $(2n+1)$-resonant. Let $B_n$ be its first non-vanishing  Birkhoff constant. If
$\mathrm{Re}(B_n)< 0$ (resp. $\mathrm{Re}(B_n)>0$), then the point $p$ is a local attractor
  (resp. repeller) point. In particular the map is not $\mathcal{C}^2$-locally integrable at $p$.
\end{lem}

\begin{proof} It suffices to prove that  $\textrm{Re}(B_k)\neq 0$ implies that the
function $V(z)=z \bar{z}=|z|^2$ is a strict Lyapunov  function at the origin for the normal form
map $F_B$ of $F$, given in~(\ref{E-FN}). For instance, when $\textrm{Re}(B_n)< 0$,
$$
V(F_B(z))=|z|^2\Big|1+2\textrm{Re}(B_n)(z\bar{z})^n+ o(|z|^{2n})\Big|<V(z),
$$
for $z$ in a small enough neighborhood of $p$, as we wanted to prove. Clearly these  maps can not
be locally integrable at $p$.~\end{proof}

As an example of computation of $B_1$, consider the elliptic fixed points of the following class
of area preserving integrable MGM maps (\cite{DR,Gu}),
\begin{equation}\label{gm}
F(x,y)=\Big(y, -x+\frac{a y}{1+y^2}\Big),\quad a\in\R.
\end{equation}

\begin{lem}\label{l-mira} The elliptic fixed points of~(\ref{gm}) are the origin, when $|a|<2$, and
$(\pm z(a),\pm z(a)),$ when $a>2,$ where $z(a)=\sqrt{(a-2)/2}.$ Moreover,

\begin{enumerate}[(i)]

\item When $|a|<2$ and $a\not\in\{-1\}$ the first Birkhoff constant of the  origin is
 \[
B_1=\frac{3 a}{\sqrt{4-a^2}}\,i.
 \]
\item When $a>2$ the first Birkhoff constant of the points $(\pm z(a),\pm z(a))$ is
\[
B_1=-\frac{4\sqrt{2}(a+4)}{a^2\sqrt{a-2}} \,i.
\]
\end{enumerate}
\end{lem}

\begin{proof} The characterization of the elliptic fixed points of~(\ref{gm}) is straightforward.

(i) The eigenvalues of the Jacobian matrix $DF(0,0)$ satisfy $\lambda^2-a\lambda+1=0$, and their
are a couple of conjugate pure imaginary values if and only if $|a|<2$. If, in addition, $a\neq
-1$ then the origin is not $3$-resonant. We introduce the rational parametrization $
a=\lambda+1/{\lambda},$ with $\lambda=\mbox{e}^{i\,\theta}, \theta\in\R\setminus\{0\},$ that
covers all the values $|a|<2$ because $a=2\,\cos\theta$. Then, the linear transformation
 $\Phi(x,y)=(x/\lambda+\lambda y,x+y)$ gives a conjugation
between $F$ and a map of the form (\ref{E-F-diagonal}). Using the expression (\ref{E-B1-Gen}),
after some computations we get
$$B_1=-\frac{3 (\lambda^2+1)}{\lambda^2-1}=\frac{3\cos \theta}{\sin\theta}\,i
=\frac{3 a}{\sqrt{4-a^2}}\,i.$$

(ii) To study the Birkhoof constants at the points $(\pm z(a),\pm
z(a))$ it is convenient in this case to introduce the new
parametrization $a=2\mu^2+2,\mu>0.$ Then the fixed points are
$(\pm\mu,\pm\mu)$  and after translating each of them  to the origin
we can proceed as in item~(i). For both points we get
\[
B_1=-\frac{2(\mu^2+3)}{\mu(\mu^2+1)^2}\,i=-\frac{4\sqrt{2}(a+4)}{a^2\sqrt{a-2}}
\,i,
\]
as we wanted to prove.~\end{proof}

The Birkhoff constants given in the above lemma are also obtained in \cite{KNP} to study the
stability of the elliptic fixed points.

Next two results are stated without detailing the proofs. The first
one follows after simple computations. The second one is as
consequence of the computations  in~\cite[Sec. 6.1]{CGM13}, because
the so called {\it periodicity conditions}, introduced and given in
that paper, essentially coincide with the Birkhoff constants.

\begin{lem}\label{L-B1-cohen}
The first Birkhoff constant of the Cohen map~(\ref{E-Cohen}) at the elliptic fixed point
$(\sqrt{3}/3,\sqrt{3}/3)$ is  $B_1=i\,135/256$.
\end{lem}

\begin{lem}\label{L-lyness}
The  first Birkhoff constant of the Lyness map $F(x,y)=(y,(a+y)/x)$, with $a\in
(-1/4,\infty)\setminus \{0,1\},$ at its elliptic fixed points is $B_1=i\,b_1\ne0,$ for some
$b_1\in\R$.
\end{lem}

In fact, it is well known that the map has elliptic fixed points only when $a>-1/4$. Moreover,
when $a=0$ (resp. $a=1$), it is 6-periodic (resp. 5-periodic) and so linearizable. Hence, when
$a\in\{0,1\}$,  $B_n=0$ for all $n\in\N.$

\section{Proof of Theorem \ref{T-Preserva-mesura}}\label{SS-prova}

One key step in the proof of Theorem \ref{T-Preserva-mesura} is
that, under its hypotheses, the map $F$ should have a \emph{Lie
Symmetry}. A vector field $X$ is said to be a Lie symmetry of a map
$F$ if it satisfies the condition
\begin{equation}\label{E-Lie-Symm}
{X}(F({x}))=DF({x})\cdot{X}({x}),
\end{equation}
where $DF$ is the jacobian matrix of $F$, \cite{CGM08,HBQC}. The vector field ${X}$ is related
with the dynamics of the map since $F$ maps any orbit of the differential system determined by the
vector field, to another orbit of this system, see \cite{CGM08}. In the integrable case, where the
dynamics are in fact one dimensional, the existence of a Lie symmetry  fully characterizes the
dynamics. In \cite[Thm. 1]{CGM08} we prove:

\begin{teo}\label{rotation}
Let  $X$ be a  $\mathcal{C}^1$-Lie Symmetry of a  $\mathcal{C}^1$-diffeomorphism $F:\U\to \U$. Let
$\gamma$ be an orbit of $X$ invariant under $F$. Then, $\left.F\right|_\g$ is the $\tau$-time map
of the flow of $X$, that is $F({x})=\varphi(\tau,{x}),\,x\in\g.$ Moreover,
\begin{itemize}
\item[(a)] If $\g\cong\{p\}$   then $p$ is a fixed point of
$F.$
\item[(b)] If $\g\cong \su$, then  $\left.F\right|_\g$ is conjugated to a rotation,
with rotation number  $\theta=\tau/T,$ where $T$ is the period of $\g$.
\item[(c)] If
$\g\cong\R$,  then  $\left.F\right|_\g$ is conjugated to a translation on the line.
\end{itemize}
\end{teo}

It can be seen that if  $F$ has a first integral $V\in \mathcal{C}^{m+1}(\U)$ on $\U\subset\R^2$
 and it  preserves a measure  absolutely continuous with respect the
Lebesgue measure with non-vanishing density $\nu\in\mathcal{C}^{m}$ in $\U$, it holds that
\begin{equation}\label{E-hamiltonian}
{X}(x,y)=\mu(x,y)\,\left(- V_y(x,y)\frac{\partial}{\partial x}+V_x(x,y) \frac{\partial}{\partial
y}\right),\quad\mbox{where}\quad \mu(x,y)=\frac 1{\nu(x,y)},
\end{equation}
is a Lie symmetry of $F$ in $\U,$ see again~\cite{CGM08}.  Observe that in  case that $F$ is an
area preserving map then $\mu\equiv 1$ and the symmetry (\ref{E-hamiltonian}) is simply the
hamiltonian vector field associated to~$V$.

We will use the following corollary of the above results.

\begin{corol}\label{cor-rot}
Let $F$ be  a $\mathcal{C}^2$-measure preserving map  with an
invariant measure with non-vanishing density
$\nu\in\mathcal{C}^{1}(\U)$ and  with a first integral $V\in
\mathcal{C}^{2}(\U)$. If $V$ has a connected component $\gamma$ of
an invariant level set $\{V( x)=h\}$, which is  invariant by $F$ and
diffeomorphic to $\su$, and $\left.F\right|_\g$ has rotation number
$\theta=\frac{\tau}{T}=p/q\in\Q,$ with $\gcd(p,q)=1$, then  $F$ has
a continuum of $q$-periodic points in $\g\subset\U$.
\end{corol}

Notice that if $p$ is a fixed point of $F,$ then from
(\ref{E-Lie-Symm}), $X(p)=DF(p)\,X(p).$ If the matrix $DF(p)$ has no
the eigenvalue $\lambda=1$ then $X(p)$ must be zero, that is, $p$ is
a singular point of the vector field.

\begin{lem}\label{P-Ls-no-degenerate}
Let $F$ be a  $\mathcal{C}^m$-planar  map that preserves an
absolutely continuous measure with respect the Lebesgue's one with
non-vanishing density $\nu\in\mathcal{C}^{m}$, with $m\geq 2$, which
is locally integrable at an elliptic fixed point $p$  with a
$\mathcal{C}^m$-local first integral $V$. Then its associated Lie
symmetry (\ref{E-hamiltonian}) has  a non-degenerate center at $p.$
\end{lem}

\begin{proof} Without loss of generality, we assume that $p=0$,
$V(0)=0$ and $V(x,y)>0$ in a neighborhood of the origin. Since $0$ is a non-degenerate critical
point of $V$ we get $ V(x,y)=ax^2+bxy+cy^2+o(|x,y|^3),$ with $=4ac-b^2>0.$

Moreover $\mu_0:=\mu(0)\neq 0$. Hence, from (\ref{E-hamiltonian}) we have that
$$
X(x,y)= -\mu_0\,(bx+2cy)+o(|x,y|^2) \frac{\partial}{\partial
x}+\mu_0\,(2a x+ by)+o(|x,y|^2) \frac{\partial}{\partial y}
$$
and therefore $\mbox{Spec}(DX(0))=\{\pm i\mu_0\,\sqrt{4ac-b^2}\neq
0\}$. Since $\{V=h\}$ are closed curves for $h>0$ small enough, we
get that $X$ has a non-degenerate center at $0.$~\end{proof}

 \begin{lem}\label{L-variacionals} Let $F$ be a  $\mathcal{C}^m$-planar  map that preserves
 an absolutely continuous measure with
respect the Lebesgue's one with non-vanishing density $\nu\in\mathcal{C}^{m}$, with $m\geq 2$,
which is locally integrable at an elliptic fixed point $p$  with a $\mathcal{C}^m$-local first
integral $V$. Let $\theta(h)$ and $T(h)$ denote, respectively, the rotation number of $F$ and the
period function  of the Lie symmetry~(\ref{E-hamiltonian}) evaluated on $\{V(x,y)=h\}$. Set
$h_p=V(p)$. Then
\begin{equation}\label{DFDX}
DF(p)=e^{\tau_p\, DX(p)},
\end{equation}
where $\tau_p=\theta_p\,T_p$, with $T_p=\lim\limits_{h\to h_p}
T(h)$, and $\theta_p=\lim\limits_{h\to h_p} \theta(h)$ is the
rotation number of the linear map $L(q)=DF(p)\,q$.
\end{lem}

\begin{proof}
Since the vector field (\ref{E-hamiltonian}) is a Lie symmetry of $F$, we know by
Theorem~\ref{rotation} that
$$
F(q)=\varphi(\tau(h),q)\mbox{ for all }q\in\U,\mbox{ where } h=V(q).
$$
By differentiation we arrive to
\begin{align*}
DF(q) &=\displaystyle{\frac{d\varphi}{dt}}(\tau(V(q)),q)\tau'(V(q))\vec{\nabla}V(q)+D_q
    \varphi(\tau(V(q)),q) \\
&     =X(F(q))\tau'(V(q))\vec{\nabla}V(q)+D_q \varphi(\tau(V(q)),q).
\end{align*}
Taking the limit as $q\to p$ and using that $X(p)=0$ we get
\begin{equation}\label{limit} DF(p)=\lim _{q\to p}\,D_q
\varphi(\tau(V(q)),q).\end{equation} Using the variational equations
we know that $M(t):=D_q \varphi(t,q),$ is the solution of
$$
        \dot{M}(t)=DX(\varphi(t,q))\, M(t),\quad
        M(0)=\,Id.
$$
Considering $q$ as a parameter in the above equation and using the
theorem of continuity respect parameters, when $q\to p,$ the
solution of the above equation tends to the solution of
$$
        \dot{M}(t)=DX(p)\, M(t),\quad
        M(0)=\,Id,
$$
that is, $\lim_{q\to p}\,D_q \varphi(t,q)$ is the fundamental matrix of the above linear system
(which has constant coefficients) that at $t=0$ is the identity. As usual, we denote this matrix
by $e^{tDX(p)}.$

Now recall that $\theta(V(q))=\tau(V(q))/T(V(q))$ and that, from Lemma~\ref{P-Ls-no-degenerate},
$p$ is a non-degenerate center, and hence there exists $\lim\limits_{q\to p}T(V(q))=T_p\neq 0$.
Moreover, from \cite[Prop. 8]{BR3}, $\lim\limits_{h\to h_p} \theta(h)=\theta_p,$ where $\theta_p$
is the  rotation number of the linear map $L(q)=DF(p)\,q$.  Hence, from (\ref{limit}) we have
$DF(p)=e^{\tau_p DX(p)},$ with $\tau_p=\lim\limits_{q\to p}\tau(V(q))=\lim\limits_{q\to
p}\theta(V(q))\,T(V(q))=\theta_p T_p,$ as we wanted to prove.~\end{proof}

\begin{lem}\label{L-simetria-isocrona}
Let $X$ be a Lie Symmetry of class $\mathcal{C}^{m}(\U)$ of a map $F$ defined in an open set
$\U\subset\R^2$, having a non-degenerate center at $p\in\U$ and let $T$ be its corresponding
period function. Then
$$
Y(x,y)=T(x,y)\,X(x,y)
$$
is a $\mathcal{C}^{m-1}(\U)$ Lie Symmetry of $F$, having an
isochronous center at $p$.
\end{lem}

\begin{proof}
From Theorem \ref{T-regularitat-periode}, we can ensure that
$Y\in\mathcal{C}^{m-1}(\U)$. A trivial computation shows that it has
a non-degenerate \emph{isochronous} center at $p$. Now, the chain of
equalities
$$
Y(F(q))={T(F(q))}\,X(F(q))={T(q)}\,X(F(q))={T(q)}\,DF(q)\,X(q)=DF(q)\,
Y(q),
$$
show that $Y$ satisfies equation (\ref{E-Lie-Symm}), so it is another Lie Symmetry of
$F$.~\end{proof}

\begin{proof}[Proof of Theorem \ref{T-Preserva-mesura}] Assume that the map
has not continua of periodic points for a sequence of unbounded positive integer numbers
$\{N_k\}_k$. Assume also that $F$ is locally integrable at $p$ with a $\mathcal{C}^{2n+4}$ first
integral $V$. By definition,  the level curves $\{V=h\}\subset \U$ contain closed curves
surrounding $p$. Since $F$ has the associated Lie Symmetry $X\in\mathcal{C}^{2n+3}(\U)$ given
by~(\ref{E-hamiltonian}), and the energy level curves $\{V=h\}$ are also integral curves of
 ${X}$, the local integrability condition also implies that the
curves $\{V=h\}$ surrounding $p$ have no singular points of $X$.

From Theorem~\ref{rotation} (b), the map $\left.F\right|_{\{V=h\}\cap \U}$ is conjugate to a
rotation with associated rotation number $\theta(h)=\tau(h)/T(h),$ where $T(h)$ is the period of
each curve $\{V=h\}$ as an orbit of ${X}$, and $\tau(h)$ is defined by the equation
$F(q)=\varphi(\tau(h),q)$, where $\varphi$ is the flow of ${X}$. This last assertion ensures, in
particular, that $\theta(h)$ is a continuous function.

Let $h_0=V(p)$ be the energy of the fixed elliptic point. It is not restrictive to assume that
$V(q)\ge h_0$ on a neighborhood of $p$.   Let us see that the proof follows by using the next
claim:

\smallskip

\noindent \emph{Claim: $\theta(h)$ is a nonconstant continuous function on a neighborhood of
$h_0$.}

\smallskip

Assuming the above claim, there is a non-degenerate rotation
interval $I=\mathrm{Image}(\theta(h))$ for $h\gtrsim h_0$, and
therefore there exists $M_1\in\N$ such that for all $N\geq M_1$
there exists $M\in\N$ coprime with $N$ such that $h_N=M/N\in I$,
with $\{V=h_N\}\cap\U\ne\emptyset$ and such that  $\theta(h_N)=M/N$.
By Corollary~\ref{cor-rot}, the set  $\{V=h_N\}\cap\U$ is full of
$N$-periodic points of $F$, in  contradiction with our hypotheses.
So $F$ is not $\mathcal{C}^{2n+4}$-locally integrable at $p$.

Now we prove the claim by contradiction. Assume that the rotation
number is a constant function $\theta(h)\equiv \theta.$ Then each
map $\left.F\right|_{\{V=h\}\cap\U}$ is conjugate to a rotation with
the same rotation number $\theta$. We will prove  that $F$ is
globally $\mathcal{C}^{2n+2}$-conjugate on $\U$ to the linear map
$L(q)=DF(p)\,q$.

From Lemma~\ref{P-Ls-no-degenerate}, $X$ has a non-degenerate center at $p$. In consequence, by
Theorem \ref{T-regularitat-periode}, the period function $T(x,y)\in\mathcal{C}^{2n+2}(\U)$ and
$T(0,0)>0$.

Now we consider  the vector field
$$
Y(x,y)={T(x,y)}\,X(x,y),
$$
which, by Lemma \ref{L-simetria-isocrona}, is again a Lie Symmetry
of $F$ of class $\mathcal{C}^{2n+2}(\U)$, having an isochronous
center at $p$ with period function identically 1.

By using the isochronicity of $Y$ and the fact that the rotation number is constant, one gets
$F(q)=\widetilde{\varphi}(\widetilde{\tau},q)$ for all $q\in \U$, where $\widetilde{\tau}$ is a
constant, and $\widetilde{\varphi}$ is the flow of $Y$.

Now we can prove that the map given by
\begin{equation}\label{E_MBmap}
\Phi(q)=\int_{0}^{1} \mathrm{e}^{-DY(p)\,s}\, \widetilde{\varphi}(s,q)\,ds,
\end{equation}
is the desired conjugation between $F$ and the linear map $L$. We remark that this linearization
is the one given in the proof of the classical Bochner Theorem (\cite{Boch} and \cite[Chap. V,
Thm. 1]{Mont}). Also notice that $\Phi\in\mathcal{C}^{2n+2}(\U)$ because of the regularity of
$\widetilde{\varphi}$. Indeed, using that $DF(p)=\mathrm{e}^{DY(p)\widetilde{\tau} },$
see~(\ref{DFDX}) in Lemma~\ref{L-variacionals},  we get
\begin{equation}\label{E_MB}
\begin{array}{rl}
\Phi\circ F(q) & =\displaystyle{\int_{0}^{1 }} \mathrm{e}^{-DY(p)\,s}\,
\widetilde{\varphi}(s,\widetilde{\varphi}(\widetilde{\tau} ,q))\,ds =\displaystyle{\int_{0}^{1 }}
\mathrm{e}^{-DY(p)\,s}\,
\widetilde{\varphi}(s+\widetilde{\tau} ,q)\,ds \\
     & =\displaystyle{\int_{\widetilde{\tau} }^{\widetilde{\tau} +1 }} \mathrm{e}^{-DY(p)(u-\widetilde{\tau} )}\,
\widetilde{\varphi}(u,q)\,du=\mathrm{e}^{DY(p)\widetilde{\tau}
}\,\displaystyle{\int_{\widetilde{\tau} }^{\widetilde{\tau} +1 }}
\mathrm{e}^{-DY(p)\,u}\, \widetilde{\varphi}(u,q)\,du\\
& \overset{(*)}{ =}\mathrm{e}^{DY(p)\widetilde{\tau} }\,\displaystyle{\int_{0}^{1 }}
\mathrm{e}^{-DY(p)\,u}\, \widetilde{\varphi}(u,q)\,du\\&=\mathrm{e}^{DY(p)\widetilde{\tau}
}\,\Phi(q)= DF(p)\, \Phi(q)=L\circ \Phi(q),
  \end{array}
\end{equation}
where in the equality marked with $(*)$ we have used that both functions $\mathrm{e}^{-DY(p)u}$
and $\widetilde{\varphi}(u,q)$ are $1$-periodic with respect the variable $u$.

Hence,  on one hand we have proved that $F$ is $\mathcal{C}^{2n+2}$-conjugate to the linear map
$L$, and on the the other hand, $F$ is $\mathcal{C}^{2n+2}$-conjugate its Birkhoff normal form
$$F_{B}(z)=\mathrm{e}^{\theta
i}\,z\,\left(1+i\,b_n\,|z|^{2n}+o(|z|^{2n})\right),
$$ where $b_n \neq 0$ and $\lambda=\mathrm{e}^{\theta i}$, a contradiction because the vanishing of
the first non-zero Birkhoff constant $B_n$  is an invariant by
$\mathcal{C}^{2n+2}$-conjugations.~\end{proof}

\section{Non-integrability of the Cohen map}\label{S-Cohen}

The proof of Theorem \ref{T-Cohen} follows as a consequence of
Theorem \ref{T-Preserva-mesura}, by using Lemma \ref{L-B1-cohen} and
the next Proposition, which states that
 there are not continua of periodic orbits of the Cohen map
for all arbitrary large period. This result  is already given
in~\cite{L}. We give a  proof based on Theorem~\ref{teog}.

\begin{propo}\label{P-Cohen}
There are finitely many $N$-periodic points for the Cohen map when
 $N\neq \dot{3}$.
\end{propo}

\begin{proof}
 The Cohen map  has the associated second order
difference equation $x_{n+2}=-x_n+\sqrt{1+x_{n+1}^2}$. An straightforward computation shows that
any $N$-periodic orbit of the Cohen map also satisfies the multivalued difference equation
\begin{equation}\label{E_multivaluada}
 g (x_n,x_{n+1},x_{n+2})=(x_n+x_{n+2})^2-x_{n+1}^2-1=0,
\end{equation}
or equivalently, the system (\ref{E-sistemaN-OP}):
$$
\left\{\begin{array}{c}
         (x_0+x_2)^2-x_1^2-1=0, \\
         (x_1+x_3)^2-x_2^2-1=0, \\
         \vdots\\
         (x_{N-2}+x_{0})^2-x_{N-1}^2-1=0, \\
         (x_{N-1}+x_{1})^2-x_{0}^2-1=0.
       \end{array}
 \right.
$$
Using Proposition \ref{L-sistemaN-OP}, there will be  a finite number of $N$ periodic orbits of
the multivalued equation (\ref{E_multivaluada}) if $\mathbf{0}$ is the unique solution of  all the
associated linear systems (\ref{E-sistemaN-OP-homo}):
$$
\left\{\begin{array}{c}
         x_0+x_2=\pm x_1, \\
         x_1+x_3=\pm x_2, \\
         \vdots\\
         x_{N-2} +x_0=\pm x_{N-1}, \\
         x_{N-1} +x_1=\pm x_0,
       \end{array}
 \right.
$$
or equivalently by setting ${\bf x}=(x_0,\ldots,x_{N-1})$, if $\mathbf{0}$ is the unique solution
of the linear systems $A_N(\varepsilon_0,\ldots,\varepsilon_{N-1}){\bf x}=\mathbf{0}$, where
$A_N(\varepsilon_0,\ldots,\varepsilon_{N-1})$ are the $N\times N$ matrices
$$
A_N(\varepsilon_0,\ldots,\varepsilon_{N-1})=\left(
                 \begin{array}{ccccccc}
                   1 & \varepsilon_0 & 1 & 0 &0 &\cdots & 0 \\
                   0 & 1 & \varepsilon_1 & 1 & 0&\cdots & 0\\
                   0 & 0 & 1 & \varepsilon_2 & 1 & \cdots & 0\\
                     &  &  &  &  &  &\\
   0 & 0 & 0 & \cdots & 1 & \varepsilon_{N-3} & 1 \\
   1 & 0 & 0 & \cdots & 0 & 1 & \varepsilon_{N-2} \\
                   \varepsilon_{N-1}  & 1 & 0 & \cdots & 0  & 0 & 1\\
                 \end{array}
\right),
$$
with $\varepsilon_j\in\{-1,1\}$, for each $j=0,\ldots,N-1$.
Proposition \ref{P-Cohen} holds if we prove that  for $N\neq
\dot{3}$ and every choice of $\varepsilon_j\in\{-1,1\}$, with
$j=0,\ldots,N-1$,
$\det(A_N(\varepsilon_0,\ldots,\varepsilon_{N-1}))\neq 0$,

To prove this fact, observe first that
\begin{equation}\label{congu}
\det(A_N(\varepsilon_0,\ldots,\varepsilon_{N-1}))\equiv \det(A_N)
\,\,\mathrm{mod} \,\,2,
\end{equation}
where $A_N:=A_N(1,\ldots,1)$. This is a consequence of the following
simple observation:  If $M=(m_{i,j})$ and $M'=(m'_{i,j})$ are two
square matrices such that $m'_{i,j}\equiv m_{i,j} \,\,\mathrm{mod}
\,\,2$, then $\det(M)\equiv \det(M') \,\,\mathrm{mod} \,\,2$.

Therefore, by \eqref{congu}, the proposition will follow once we
show that
\begin{equation}\label{tres}
a_n=\det(A_N)=\left\{
            \begin{array}{ll}
              3 & \hbox{if } N\neq \dot{3}, \\
              0 & \hbox{if } N=\dot{3}.
            \end{array}
          \right.
\end{equation}
To prove \eqref{tres} we introduce $t_n=\det(T_n[1,1])$ where
$T_n[1,1]$ is the $n\times n$ Toepliz matrix
$$
T_n[1,1]=\left(
                 \begin{array}{ccccccc}
                   1 & 1 & 0 & 0 &0 &\cdots & 0 \\
                   1 & 1 & 1 & 0 & 0&\cdots & 0 \\
                   0 & 1 & 1 & 1 & 0 & \cdots &0 \\
                     &  &  &  &  &  &\\
                     &  &  &  &  &  &\\
                   0 &  & \cdots & 0 & 1 & 1 & 1 \\
                   0  &  & \cdots & 0 &  0 & 1 & 1\\
                 \end{array}
\right).
$$
We will use the next claims, which are inspired in the results in
\cite{CCR}:

\noindent \emph{Claim 1: The sequence $a_n$ satisfies $a_n=(-1)^{n-1}t_{n-1}+2(-1)^{n}t_{n-2}+2$.}

\noindent  \emph{Claim 2: The sequence $t_n$ satisfies $t_n=t_{n-1}-t_{n-2}$, with $t_1=1$ and
$t_2=0$, and therefore it is the $6$-periodic sequence $\{1,0,-1,-1,0,1\}$.}

By using them, a straightforward computation shows that $a_{n+6}=a_n$ and therefore $a_n$ is
$6$-periodic. Finally, since $a_3=0$, $a_4=3$, $a_5=3$, $a_6=0$, $a_7=3$ and $a_8=3$, $a_n$ is
$3$-periodic and the equality \eqref{tres} holds.

Now we prove Claim 1: Let $M_{i,j}$ be the $(i,j)$-minor of $A_n$.
By using the Laplace expansion of the last row of $a_n$ we get
$$
a_n=(-1)^{n+1}M_{n,1}+(-1)^{n+2}M_{n,2}+(-1)^{2n}M_{n,n}=
(-1)^{n+1}M_{n,1}+(-1)^{n}M_{n,2}+M_{n,n}.
$$

It is straightforward to check that
$M_{n,1}=\det(T_{n-1}[1,1])=t_{n-1}$. Observe that
$$
M_{n,2}=\det\left(
                 \begin{array}{ccccccc}
                   1 & 1 & 0 & 0 & 0&\cdots & 0 \\
                   0 & 1 & 1 & 0 & 0&\cdots & 0 \\
                   0 & 1 & 1 & 1 & 0 & \cdots &0 \\
                     &  &  &  &  &  &\\
                                          &  &  &  &  &  &\\
                   0 &  & \cdots & 0 & 1 & 1 & 1 \\
                   1  &  & \cdots & 0 &  0 & 1 & 1\\
                 \end{array}
\right)
 \mbox { and }\,
M_{n,n}=\det\left(
                 \begin{array}{ccccccc}
                   1 & 1 & 1 & 0 & 0&\cdots & 0 \\
                   0 & 1 & 1 & 1 & 0&\cdots & 0 \\
                   0 & 0 & 1 & 1 & 1 & \cdots &0 \\
                     &  &  &  &  &  &\\
                                          &  &  &  &  &  &\\
                   0 & 0 & \cdots &  & 0 & 1 & 1 \\
                   1  & 0 & \cdots &  &  0 & 0 & 1\\
                 \end{array}
\right).
$$
By using again the Laplace expansion on the first column of
$M_{n,2}$  we get that $M_{n,2}=$ $\det\left(T_{n-2}[1,1]\right)$
$+(-1)^n\det\left(L_{n-2}\right),$ where $L_{n-2}$ is a lower
triangular $(n-2)\times (n-2)$ matrix such that all the diagonal
entries are $1$. Hence $M_{n,2}=t_{n-2}+(-1)^n$. Analogously,
$M_{n,n}=\det\left(U_{n-2}\right)+(-1)^n
\det\left(T_{n-2}[1,1]\right)$, where $U_{n-2}$ is an upper
triangular matrix such that all the diagonal entries are $1$.
Therefore, $M_{n,n}=1+(-1)^n t_{n-2}$.  Hence $$\begin{array}{rl}
                                      a_n &  =
(-1)^{n+1}t_{n-1}+(-1)^{n}\left(t_{n-2}+(-1)^n\right)+1+(-1)^n
t_{n-2} \\
                                        & =(-1)^{n-1}t_{n-1}+2(-1)^{n}t_{n-2}+2,
                                     \end{array}
$$ and the claim is proved.

Claim 2 follows by applying once again the Laplace expansion of $t_n$,  obtaining that it
satisfies the linear difference equation $t_n=t_{n-1}-t_{n-2}$ with initial conditions $t_1=1$ and
$t_0=0$. It is a simple computation to check that it is $6$-periodic. ~\end{proof}

\begin{nota} A different proof of Proposition~\ref{P-Cohen}  follows using
previous results that involve the celebrated  Fibonacci numbers
$F_N$. From  \cite[p.78]{CCR} it holds that $\mathrm{per}(A_N),$ the
permanent of $A_N,$ satisfies $\mathrm{per}(A_N)=F_{N}+2F_{N-1}+2$,
where the permanent of a $n\times n$ matrix $M=(m_{i,j})$ is given
by $\mathrm{per}(M)= \sum_{\sigma\in\Sigma_n} \prod_{i=1}^n
m_{i,\sigma(i)}$. Since $\mathrm{per}(M)\equiv \det(M)\,\,
\mathrm{mod}\,\,2$, see for instance \cite{S}, we know that
$\det(A_N)\equiv F_{N}+2F_{N-1}+2\,\,\mathrm{mod}\,\,2$. Finally,
the Fibonacci numbers ($\mbox{mod } 2$) are $1,1,0,1,1,0,\ldots,$
giving our desired result.\end{nota}

\section{Other applications}\label{S-Other}

\begin{proof}[Proof of Theorem~\ref{P_racionals2}] We want to apply
Theorem~\ref{T-Preserva-mesura}. Let $p$ be  an elliptic fixed point of $F$, not
$(2n+1)$-resonant, and such that its first non-vanishing Birkhoff constant is  $B_n$ is purely
imaginary. We already know that  $F$ preserves the measure with density $\nu(x,y)=1/(xy)$, which
is analytic in a neighborhood of $p$. So,  it only remains to check whether the hypothesis about
the finiteness of periodic points of $F$ holds.

Observe that any $N$-periodic point of $F$ is a periodic orbit of the second order recurrence
\begin{equation}\label{recuPQ}
 g (x_n,x_{n+1},x_{n+2})=x_{n+2}x_nQ(x_{n+1})-P(x_{n+1})=0,
\end{equation}
where $f=P/Q$. Hence the $N$-periodic points are the solutions of
$$
\left\{\begin{array}{c}
         x_{n+2}x_nQ(x_{n+1})-P(x_{n+1})=0\,\,,\,\,n=0,1,\ldots,N-3, \\
         x_0x_{N-2}Q(x_{N-1})- P(x_{N-1})=0,\\
         x_1x_{N-1}Q(x_0)- P(x_0)=0.
       \end{array}
 \right.
$$
Setting $P(x)=\sum_{j=0}^{p} a_j x^j$ and $Q(x)=\sum_{j=0}^{q} b_j
x^j$,  since $\deg(f)=p-q>2$, the system~\eqref{E-sistemaN-OP-homo}
associated to the above one is

\begin{equation*}
a_p\,x_i^p=0,\mbox{ for }i=0,1,\ldots,N-1.
\end{equation*}
and $\mathbf{0}$ is its unique solution. Then,  by Proposition~\ref{L-sistemaN-OP}, we have that
for each $N$ there are finitely many $N$-periodic points of $F$. Therefore all the hypotheses of
Theorem~\ref{T-Preserva-mesura} hold and, as a consequence, $F$ is not
$\mathcal{C}^{2n+4}$-locally integrable at $p$, as we wanted to prove.~\end{proof}

It is interesting to notice that when $\deg(f)=p-q<2$   Proposition \ref{L-sistemaN-OP} never
applies. Indeed, in this case, using the notation introduced in~(\ref{Notacio}),
$$
 g _{i,d}(x_{i+2\mbox{ mod } N},x_{i+1\mbox{ mod } N},x_{i\mbox{
mod } N})=b_q\, x_{i\mbox{ mod } N}\cdot x_{i+2\mbox{ mod } N}\cdot
x_{i+1\mbox{ mod } N}^q,
$$
so $\mathbf{0}$ is not an isolated zero of system (\ref{E-sistemaN-OP-homo}). As we have already
explained in the introduction, in this subfamily there are several integrable cases.

When $\deg(f)=p-q=2$, then $$ \begin{array}{rl}  g _{i,d}(x_{i+2\mbox{ mod } N},x_{i+1\mbox{ mod }
N},x_{i\mbox{ mod } N})=&b_q\,x_{i\mbox{ mod }
N}\cdot x_{i+2\mbox{ mod } N}\cdot x_{i+1\mbox{ mod } N}^q\\
{}&-a_p\,x^p_{i+1\mbox{ mod } N}. \end{array}$$ and  there are examples where Proposition
\ref{L-sistemaN-OP} applies and other where it does not. The same happens with
Theorem~\ref{T-Preserva-mesura}, as the proof of Proposition~\ref{familybracfacil} shows.

\begin{proof}[Proof of Proposition~\ref{familybracfacil}]
 First notice that when $C=1$ it is easy to check that the function $V$ given in the statement is
in fact a first integral of the map. So we proceed with $C\ne1.$ In this case we also want  to
apply Theorem~\ref{T-Preserva-mesura} for proving local non-integrability.  We already know that
$F$ preserves the measure with density $\nu(x,y)=1/(xy)$. To continue our proof, we have to show
that  $F$ has an elliptic fixed point with suitable non-zero Birkhoff constant and moreover that
for an unbounded sequence of natural number $\{N_k\}_k$,  $F$ has not continua of $N_k$-periodic
points in a neighborhood of this point. We start proving this second fact.

Each map (\ref{BRgen}) has the associated difference equation
$$
 g (x_n,x_{n+1},x_{n+2})=x_{n+2}x_n-(A+Bx_{n+1}+Cx_{n+1}^2)=0,
$$
with corresponding system (\ref{E-sistemaN-OP-homo}):
\begin{equation}\label{sistemafamBRac}
\left\{\begin{array}{l}
         x_0\,x_2-C\,x_1^2=0, \\
         x_1\,x_3-C\,x_2^2=0, \\
         \vdots\\
         x_{N-3}\,x_{N-1}-C\,x_{N-2}^2=0, \\
         x_{N-2}\,x_0-C\,x_{N-1}^2=0, \\
          x_{N-1}\,x_1-C\,x_0^2=0.
       \end{array}
 \right.
\end{equation}
Observe that  $\mathbf{0}$ is always a solution of
(\ref{sistemafamBRac}). Let us prove  that since $C\neq 1$,  there
are no other solutions. Assume that the system has some non-zero
solution $(y_0,\ldots,y_{N-1})\neq\mathbf{0}$. Straightforward
computations show that then $y_i\neq 0$ for all $i=0,\ldots,{N-1}$.
Moreover, from system (\ref{sistemafamBRac}) one easily gets
$$
C\, (y_0\cdots y_{N-1})^2=(y_0\cdots y_{N-1})^2\neq 0,
$$
fact that is in contradiction with  $C\neq 1$. Therefore, by Proposition~\ref{L-sistemaN-OP}, we
have proved that when $C\ne1$,  for each $N\in\N$ the map has finitely many $N$-periodic points.

Finally, let us study the existence of elliptic fixed points and compute their corresponding
Birkhoff constants when $C\ne1$.

To facilitate the computations we introduce the change of variables $u=\alpha x$, $v=\alpha y$
where $A\alpha^2+B\alpha+(C-1)=0$. It conjugates $F$ with the map
\begin{equation}\label{BRgen2}
   \bar{F}(x,y)=\left(y,\frac{a+(1-a-c)y+cy^2}{x}\right), \quad c\ne0,
\end{equation}
where $a=\alpha^2 A$, $c=C$. Observe that if $F$ has a simple fixed  point then $B^2-4A(C-1)>0$,
so $\alpha$ is actually a  real number. With this change of variables any fixed point of our
initial $F$ is brought to the fixed point $p=(1,1)$ of this new $F$, which is elliptic if and only
if
\begin{equation}\label{elip}
 -1<a-c<3.
\end{equation}

  Introducing the parameter $r^2=(3-a+c)/(a+1-c),$ it is easy to see that $r$ is a
real number and that the eigenvalues associated to $p$ are $\lambda=(r^2-1\pm 2 r\,i)/(r^2+1).$
Straightforward computations give:
\begin{itemize}
\item $\lambda=1$ if and only if $a-c=-1.$
\item $\lambda^2=1$ with $\lambda\ne 1$, or equivalently, $r=0$ if and only if $a-c=3.$
\item $\lambda^3=1$ with $\lambda\ne 1$ if and only if $r^2=\frac{1}{3}$, or equivalently, $a-c=2.$
\item $\lambda^4=1$ with $\lambda^2\ne 1$ if and only if $r=1$, or equivalently, $a-c=1.$
\item $\lambda^5=1$ with $\lambda\ne 1$ if and only if $r^2=1\pm 2\sqrt{5}/5$, or equivalently,
 $a-c=(3\mp\sqrt{5})/2.$
\end{itemize}
Therefore, 1 or 2-resonances can not appear in this new map.

After a change of variables bringing $p$ at the origin, and
computing the first Birkhoff constant, see the equation
(\ref{E-B1-Gen}), we get
$$B_1=
\frac {i\,Q_1(a,c)\,(1+r^2)^3}{16\,r\,(1-3\,r^2)},$$ where
$$
Q_1(a,c):={a}^{4}-3\,{a}^{3}c+3\,{a}^{2}{c}^{2}-a{c}^{3}-4\,{a}^{3}+5\,{a}^{2}c-
2\,a{c}^{2}+{c}^{3}+4\,{a}^{2}+4\,ac-2\,{c}^{2}-a+c.
$$

When $\lambda^3\ne1$ and $Q_1(a,c)\ne 0,$ from Theorem \ref{T-Preserva-mesura}, the map is not
$\mathcal{C}^6$-locally integrable at the elliptic point.

Now assume that $Q_1(a,c)=0$ and consider the second Birkhoff constant. When $(\lambda^4-1)(
\lambda^5-1)\ne0,$ tedious calculations using the formula given in~\cite[App. A]{CGM13}, lead to
$$B_2=\frac{(1+r^2)^8\,(r+i)^2\,i\,Q_2(a,c)}{r^3\, (3r^2-1)^3\, (r^2-1)},$$
where $Q_2(a,c)$ is a real polynomial of degree $11,$ that we omit
for the sake of simplicity.
In fact, the above mentioned computations show that when
$Q_1(a,c)\ne 0,$ then $B_2$ has real part different from $0$, and
therefore by Proposition~\ref{L-dynamicaFN-gen}, the elliptic point
is attractor or repeller and the map is not $\mathcal{C}^2$-locally
integrable at this point.

If $Q_1(a,c)=0$ and $Q_2(a,c)=0$ we have to deal with  a finite
number of values of $a$ and $c.$ All these values, when $c\neq 0$,
are:
\[
(a,c)\in\Big\{(4,1),\,(2-1),\,(3,1),\,(0,1),\, \Big(\frac{5\pm
3\sqrt 5}{4},\frac{\pm\sqrt 5-1}{4} \Big) \Big\}.
\]
The first two cases are not under condition~(\ref{elip}) and hence
do not give rise to elliptic fixed points. The third and fourth ones
are inside the integrable case $c=1.$ Finally, for the last two
pairs, the corresponding eigenvalues are
$$\lambda=\frac{\sqrt{5}}{\sqrt{5}-5}\pm\frac{\sqrt{25-10\,\sqrt{5}}}{\sqrt{5}-5}\,i$$
which satisfy $\lambda^5=1.$ These cases correspond to 5-resonances and they are not covered by
our approach. ~\end{proof}

\begin{nota} From the proof of the above theorem we get a slightly stronger result. If
$\lambda^3\ne1$ and $B_1\ne 0$, the map is not $\mathcal{C}^6$-locally integrable at the elliptic
fixed point, still if $\lambda^k=1$ for $k=4$ or $5$.\end{nota}

\begin{proof}[Proof of Theorem~\ref{P_racionals3}]
Its proof is similar to the one of Theorem~\ref{P_racionals2} and we
omit it for the sake of shortness. In this case, the condition
$\deg(f)>1$ is the one that allows to apply
Proposition~\ref{L-sistemaN-OP} to ensure that, for each $N\in\N$,
the map~(\ref{apm}) has finitely many $N$-periodic points, and then
apply Theorem~\ref{T-Preserva-mesura}.
\end{proof}

The following remark shows a relation between the two families of maps~\eqref{fyx}
and~\eqref{apm}.

\begin{nota} Consider the diffemorphism $\Psi:\R^2\rightarrow (\R^+)^2$,
$\Psi(x,y)=(\operatorname{e}^x,\operatorname{e}^y)$ with inverse $\Psi^{-1}(x,y)=(\log x,\log y)$.
It holds that
\begin{enumerate}[(i)]

\item  If $F(x,y)=(y,f(y)/x)$ then  $\Psi^{-1}\circ F \circ \Psi(x,y)=\big(y,-x+\log f
\big(\operatorname{e}^y\big)\big)$.

\item If $F(x,y)=(y,-x+f(y))$ then  $\Psi\circ F \circ \Psi^{-1}(x,y)=
\Big(y,\dfrac{\operatorname{e}^{f(\log y)}}x\Big)$.
\end{enumerate}
Therefore, Theorems~\ref{P_racionals2} and~\ref{P_racionals3} can also be applied  to some
non-rational maps of the forms~\eqref{fyx} or~\eqref{apm}. More concretely, the maps  such that,
via the changes of variables $\Psi$ or $\Psi^{-1}$, can be transformed into rational maps of the
other family.
\end{nota}

\begin{proof}[Proof of Corollary~\ref{pertur}] (i)  Consider the integrable
MGM map $F_0$. By Lemma~\ref{l-mira}, when  $a\in
(-2,\infty)\setminus \{-1,0,2\}$ it has at least one elliptic fixed
point, say $p_0$, such that it is not $3$-resonant and with non-zero
purely imaginary first Birkhoff constant $B_1$. When
$0\ne|\varepsilon|$ is small enough, by continuity with respect to
$\varepsilon$ and since $F_\varepsilon$ is area preserving, the map
$F_\varepsilon$  also has  an elliptic fixed point, say
$p_\varepsilon$, satisfying the same properties, that is, being not
$3$-resonant and with non-zero purely imaginary first Birkhoff
constant $B_1(\varepsilon)$. Moreover, since
\[
\deg\Big( \frac{a y}{1+y^2}+\varepsilon\,g(y)\Big)=\deg(g(y))>1,
\]
we are under the hypotheses of Theorem~\ref{P_racionals3} for $n=1$ and, therefore,
$F_\varepsilon$ is not $\mathcal{C}^{6}$-locally integrable at $p_\varepsilon,$ as we wanted to
prove.

(ii) Its proof follows exactly the same steps that the proof of item (i), changing
Lemma~\ref{l-mira} and Theorem~\ref{P_racionals3}, by the corresponding results
Lemma~\ref{L-lyness} and Theorem~\ref{P_racionals2}.~\end{proof}

\begin{proof}[Proof of Corollary~\ref{pertur2}] When $\varepsilon=0$  the result follows
from Theorem \ref{T-Cohen}. When $\varepsilon\ne0$, the proof follows the same scheme  of the one
of Corollary~\ref{pertur}. The main difference is the way we show that, for each $N\in\N$,
$H_\varepsilon$ has finitely many $N$-periodic orbits. Similarly that in the  case
$\varepsilon=0$, we have to study the number of solutions of
$$
\left\{\begin{array}{c}
         (x_0+x_2-\varepsilon g(x_1))^2-x_1^2-1-\varepsilon h(x_1)=0, \\
         (x_1+x_3-\varepsilon g(x_2))^2-x_2^2-1-\varepsilon h(x_2)=0, \\
         \vdots\\
         (x_{N-2}+x_{0}-\varepsilon g(x_{N-1}))^2-x_{N-1}^2-1-\varepsilon h(x_{N-1})=0, \\
         (x_{N-1}+x_{1}-\varepsilon g(x_0))^2-x_{0}^2-1-\varepsilon h(x_0)=0.
       \end{array}
 \right.
$$
Writing $g=P/Q$ and $h=R/S$ with $P(x)=\sum_{j=0}^{p} a_j x^j$,  $Q(x)=\sum_{j=0}^{q} b_j x^j,$
$R(x)=\sum_{j=0}^{r} c_j x^j$,  $S(x)=\sum_{j=0}^{s} d_j x^j$, the above set of rational equations
can be transformed into a polynomial one. Recall that $\deg(g)=p-q$ and $\deg(h)=r-s.$ Assume for
instance that $\deg(g)>1$ and $2\deg(g)>\deg(h)$; the other cases can be studied similarly. It
holds that the system of equations \eqref{Notacio} associated to the above one is
\[
d_s a_p^2 x_j^{2p+s}=0, \mbox{ for } j=0,1,\ldots N-1,
\]
which trivially has the only solution ${\bf x}={\bf 0}$. Hence by
Proposition \ref{L-sistemaN-OP}, the map $H_\varepsilon$ has
finitely many $N$-periodic orbits, as we wanted to prove.
\end{proof}

\begin{nota} When in Corollary~\ref{pertur2} it holds that $\deg(h)=2\deg(g)>2$, by using
the same approach we obtain that the same result holds when $d_s a_p^2\ne c_r b_q^2.$
\end{nota}

\end{document}